\documentclass[12pt]{amsart}
\usepackage{graphicx} 

\usepackage[english]{babel}
\usepackage{amsthm}
\usepackage{amssymb}
\usepackage{amsmath}
\usepackage{pdfpages}
\newcommand{\dpw}{d_\mathcal{C^*}}	
\usepackage{pinlabel}
\usepackage{tikz}

\newcommand{\dpants}{d_{\mathcal{P}}}

\newtheorem{theorem}{Theorem}[section]

\newtheorem{proposition}[theorem]{Proposition}
\newtheorem{lemma}[theorem]{Lemma}
\newtheorem{remark}[theorem]{Remark}
\newtheorem{example}[theorem]{Example}

\newtheorem{definition}[theorem]{Definition}

\usepackage{pinlabel} 
\usepackage{tikz}

\usepackage{biblatex}
\makeatletter
\@namedef{subjclassname@2020}{%
  \textup{2020} Mathematics Subject Classification}
\makeatother
\makeatletter
\def\moverlay{\mathpalette\mov@rlay}
\def\mov@rlay#1#2{\leavevmode\vtop{%
   \baselineskip\z@skip \lineskiplimit-\maxdimen
   \ialign{\hfil$\m@th#1##$\hfil\cr#2\crcr}}}
\newcommand{\charfusion}[3][\mathord]{
    #1{\ifx#1\mathop\vphantom{#2}\fi
        \mathpalette\mov@rlay{#2\cr#3}
      }
    \ifx#1\mathop\expandafter\displaylimits\fi}
\makeatother

\title[Simplicial complexes with applications to 3-manifolds]{Estimating distances in simplicial complexes with applications to 3-manifolds and handlebody-knots}

\author{Sayantika Mondal}
\address[Sayantika Mondal]{The Graduate Center, CUNY \\ 365 Fifth Ave., N.Y., N.Y., 10016}
\email{smondal@gradcenter.cuny.edu}
\thanks{}
\author{Puttipong Pongtanapaisan}
\address[Puttipong Pongtanapaisan]{Arizona State University}
\email{ppongtan@asu.edu}
 \thanks{}
 \author{Hanh Vo}
\address[Hanh Vo]{Arizona State University}
\email{hanhmfa@gmail.com}
 \thanks{}
\keywords{handlebody-knots, simplicial complexes, 3-manifolds}
\subjclass[2020]{Primary 57M25; Secondary 57M27, 57M50}

\usepackage{geometry}
\usepackage{hyperref}

\begin{document}

\begin{abstract}
 We study distance relations in various simplicial complexes associated with low-dimensional manifolds. In particular, complexes satisfying certain topological conditions with vertices as simple multi-curves. We obtain bounds on the distances in such complexes in terms of number of components in the vertices and distance in the curve complex.
   
We then define new invariants for closed 3-manifolds and handlebody-knots. These are defined using the splitting distance which is calculated using the distance in a simplicial complex associated with the splitting surface arising from the Heegard decompositions of the 3-manifold. We prove that the splitting distances in each case is bounded from below under stabilizations and as a result the associated invariants converge to a non-trivial limit under stabilizations.
\end{abstract}

\maketitle


\section{Introduction}   
There has been a great deal of interest in simplicial complexes such as the curve complex and the pants complex due to their applications to mapping class groups of surfaces and Teichm\"{u}ller theory. Since low-dimensional manifolds and knotted objects inside them can be decomposed into simple pieces along a surface $\Sigma$, one can use the aforementioned complexes on $\Sigma$ to measure the entanglement complexities.

Previously, Campisi and Rathbun studied knotted graphs using arc and curve complex \cite{campisi2018hyperbolic} on the splitting surface. In particular, if the distance in the complex is greater than 3, then the exterior of the knotted graph is hyperbolic. The complexity discussed in this paper provides a more suitable measure for splittings with low arc and curve distance. In particular, stabilizing the splitting can cause the arc and curve distance to drop drastically, but not the complexities we calculate.
Nevertheless, our complexities can be estimated using the arc and curve complex.

Our general strategy for defining a 3-manifold invariant involves starting with a Heegaard splitting of our 3-manifold, then defining a distance $D(\Sigma)$ for the splitting surface using a simplicial complex and then using this distance to define a complexity measure. And finally showing that the sequence of complexities converges under stabilizations. We show this by proving that the splitting distance is bounded below. Thus, having universal lower bounds for splitting distances defined using various complex would be useful. 

Our main results in Section \ref{section relation distance} finds lower bounds on distance in various simplicial complexes associated with a surface in terms of the number of components of each vertex of the 1-skeleton of the complex and minimum distance in curve complex between the components of the vertices.

Let $\Sigma_{g,n}$ be a finite type surface with genus $g$ and $n$ punctures. We consider simplicial complexes associated with this surface, where for some fixed natural number $N$, the vertices are simple multi-curves on the surface, with each multi-curve consisting of $N$ components (disjoint simple closed curves) and the edges corresponding to some set of \textit{admissible} moves. An admissible move corresponds to replacing one simple curve component of the vertex with another satisfying the conditions that the new curve intersects the one it is replacing and is disjoint from the other, in addition to other conditions. We call such a simplicial complex an \textit{admissible multi-curve complex}.

We show that for any such simplicial complex, the following result holds.

\vspace{10pt}

\noindent \textbf{Theorem A} (Theorem \ref{general complex distance}): \textit{Let $V_1$ and $V_2$ be two vertices of an admissible multi-curve complex.
 For all $k\ge1$, if each loop of $V_1$ has a curve complex distance of at least $k$ from any loop of $V_2$, then 
 \[
 d_{cw} (V_1,V_2) \geq k (N -1) + 1.
\]}

  where $d_{cw}$ is the distance in the 1-skeleton of the simplicial complex. 

\vspace{10pt}

We prove a special case of Theorem A in Theorem \ref{prop dC and dP}, where the simplicial complex is the dual curve complex. In this case, we also provide a geometric proof (See Proposition \ref{proposition general formula flipping}).
\vspace{10pt}

\noindent \textbf{Theorem B} (Theorem \ref{prop dC and dP}):
\textit{Let $P$ and $T$ be pants decompositions. 
For all $k\ge1$, if each curve of $P$ has a curve complex distance at least $k$ from any curve of $T$ then 
\[\dpw (P,T) \geq k (N -1) + 1.
\]}
 where $\dpw$ is the distance in the dual curve complex and $N$ is the number of curves in a pants decomposition.
\vspace{10pt}

Johnson in \cite{johnson2006heegaard} introduced two notions of distances for a splitting surface and integral measures of complexity for a 3-manifold using the dual curve complex and the pants complex. And showed that as the Heegaard splitting is stabilized, the sequence of complexities converges to a non-trivial limit depending only on the manifold. We define a similar splitting distance $D^{HT}(\Sigma_g)$ using the Hatcher-Thurston cut system, where the distance is the minimum over cut systems that define the two handlebodies in the Heegaard splitting and define a complexity $A^{HT}_g(\Sigma) = D^{HT}(\Sigma_g)-g-n$.  
Here $n$ is the number of $S^1\times S^2$ components of the prime decomposition of the manifold.
Defining a handlebody means the curves of the vertex set bounds disks in the handlebodies. 
We show this sequence of complexities converge under stabilizations. 

\textbf{Theorem C} (Theorem \ref{Hatcher-Thurston}): \textit{The limit $\lim_{g\rightarrow\infty} A^{HT}_g(\Sigma)$ exists and is a 3-manifold invariant.}

Ozawa in \cite{ozawa2021stable} defined a notion of stable equivalence for bridge-positions of handlebody-knots. We define a stable invariant for knotted handlebodies, in a similar spirit as that for Heegaard splittings of 3-manifolds using a splitting distance defined in terms of dual curve distance. The main difference being we allow the pants curves to now bound punctured disks or disks in the handlebodies. We define the complexity measure to be $$B_\Sigma(c,s_1,s_2) = \frac{1}{4}D(\Sigma_{c,s_1,s_2}) - \frac{s_1+s_2}{4}$$ where $s_1 , s_2$ are the number of $S_1$ and type $S_2$ stabilization moves respectively (see Figure \ref{fig:2moves} for the two types of moves) and $c$ is number of punctures of the splitting surface.  $D(\Sigma_{c,s_1, s_2})$ denotes the dual curve distance of a bridge sphere, that is the minimum of the dual curve graph distance between pants decompositions that define the two handlebodies corresponding to a Heegaard slitting along a $c$-punctured sphere, respectively, after $s_1$ and $s_2$ moves. 

\textbf{Theorem D} (Theorem \ref{main invariant}): \textit{ The limit $\lim_{s_1\to\infty, s_2\to\infty} B_\Sigma(c,s_1,s_2)$ exists and is an invariant.}

\subsection{Section Overview}

We begin by introducing various structures associated with 2 and 3-manifolds in Section \ref{Background}. Various simplicial complexes that can be associated with an orientable surface with simple multicurves on the surface as vertices, are introduced in Section \ref{Simplicial complexes} and Section \ref{Distances in simplicial complexes}.  
In Section \ref{3 mani} and \ref{dist 3 mani} we discuss 3-manifold decompositions, and associated distances and complexity measures.

In Section \ref{section relation distance}, we study relations between distances in various complexes and their relations to distances in the curve complexes. In particular, we find lower bounds for distances in any complex in terms of the number of components of vertices and minimum distance of the vertices in the curve complex.

We next use these bounds to study the complexities for 3-manifolds in Sections \ref{Applications to 3-manifolds}. In Section \ref{Hatcher-Thurston}, we study a complexity invariant defined using the Hatcher-Thurston cut system for closed 3 manifolds. In Section \ref{section handle body links}, we study a complexity associated with handlebody-knots. In both cases, the complexity is defined using distances in simplicial complexes on the splitting surface corresponding to a Heegard splitting of the 3-manifold. We get immediate lower bounds on the distance and hence the complexity using results from Section \ref{section relation distance}, these bounds are in fact stronger, as long as the minimum distance is greater than 1, but it remains to be seen if this condition always holds. In the case where the minimum distance is $0$ this gives us no information, hence we compute alternate bounds on the distance and complexity measures that are independent of $k$.

We end with some open questions in Section \ref{open questions}.

\subsection{Notation}
In the following table, we list commonly used notations and the first section they appear in.

\begin{center}\label{Table:notation}
	\begin{tabular}{| l | c | c |}
		\hline \hline 
		{\bf Definition} & {\bf Section} &{\bf Notation} \\
		\hline
		\hline 
		\hline
	Curve complex& 1 &$C(\Sigma)$ \\
	\hline
         Distance in curve complex&2 &$d(\cdot,\cdot)$ \\
        \hline
        Dual curve complex & 2 & $C^*(\Sigma)$\\
        \hline
        Dual curve complexity& 4 & $A_g(\Sigma)$ \\
        \hline
        Dual curve distance &2 &$\dpw(\cdot,\cdot)$ \\
        \hline
        Dual distance of $\Sigma$ & 2 &$D(\Sigma)$ \\
        \hline
        Handlebody-knot complexity& 1 & $B_{\Sigma}(c,s_1, s_2)$\\
        \hline
	  Hatcher-Thurston cut system  & 2& $C^{HT}(\Sigma)$ \\
      \hline
      Hatcher-Thurston cut system complexity&4&$A^{HT}_g(\Sigma)$\\
	\hline
        Heegaard splitting  &2 &$(\Sigma, H_1, H_2$) \\
        \hline
        Hempel distance & 2 & $d(\Sigma)$ \\
        \hline
        Pants complex& 1 & $\mathcal{P}(\Sigma)$\\
        \hline
        Pants distance & 2 & $\dpants(\cdot,\cdot)$ \\
        \hline
        Pants graph complexity& 4 & $A_g^{\mathcal{P}}(\Sigma)$\\
        \hline
        Pants distance of $\Sigma_g$ & 4 & $D^{\mathcal{P}}(\Sigma_g)$\\
        \hline
		
		\hline
		\hline
	\end{tabular} 
\end{center}

\subsection{Acknowledgement}
We would like to thank Ara Basmajian, Tommaso Cremaschi and Julien Paupert for support and feedback. Hanh Vo is supported by the AMS-Simons Travel Grant.

\section{Background}\label{Background}

\subsection{Simplicial complexes}\label{Simplicial complexes}
 Let $\Sigma$ be a
 connected, orientable surface of negative Euler charateristic.  
 Then, we can define the following simplicial complexes, whose 1-skeletons are metric spaces. We will often refer to the complexes simply by the graph that gives their 1-skeleton.
 
 \subsubsection{Curve complex}
 
The curve complex $C(\Sigma)$ is the cell complex defined as follows: The vertices of $C(\Sigma)$ are isotopy classes of non-trivial, simple closed curves in $\Sigma$. An edge connects two vertices if and only if there are representatives of the two isotopy classes
which are disjoint.
We can attach cells of higher dimensions to the graph. A collection of vertices
$\{u_0, \cdots , u_n\}$ bounds an $n$–simplex if and only if these vertices are pairwise disjoint.
 
 \subsubsection{Dual curve complex}
The \emph{dual curve complex} $C^*(\Sigma)$ is the graph whose vertices correspond to maximal simplices in the curve complex $C(\Sigma)$—that is, pants decompositions of $\Sigma$. Two vertices in $C^*(\Sigma)$ are connected by an edge if their corresponding pants decompositions differ by replacing a single curve with another that is disjoint from the rest, resulting in a new pants decomposition.

\subsubsection{Pants complex} 
A \emph{pants decomposition} of $\Sigma$ is a maximal collection of disjoint, non-parallel, essential simple closed curves such that cutting along them decomposes $\Sigma$ into three-holed spheres (also called pairs of pants). 
The \emph{pants complex} $\mathcal{P}(\Sigma)$ is a graph defined as follows:
\begin{itemize}
  \item \textbf{Vertices:} Each vertex corresponds to a pants decomposition of $\Sigma$.
  \item \textbf{Edges:} Two vertices are connected by an edge if the corresponding pants decompositions differ by an \emph{elementary move}: replacing a single curve in the decomposition with another disjoint curve such that the result is again a pants decomposition (see Figure \ref{pants-moves}). The replacement must occur within a subsurface of complexity 1 (i.e., a four-holed sphere or a one-holed torus).
\end{itemize}

\begin{figure}[h]
    \centering \includegraphics[width=0.40\linewidth]{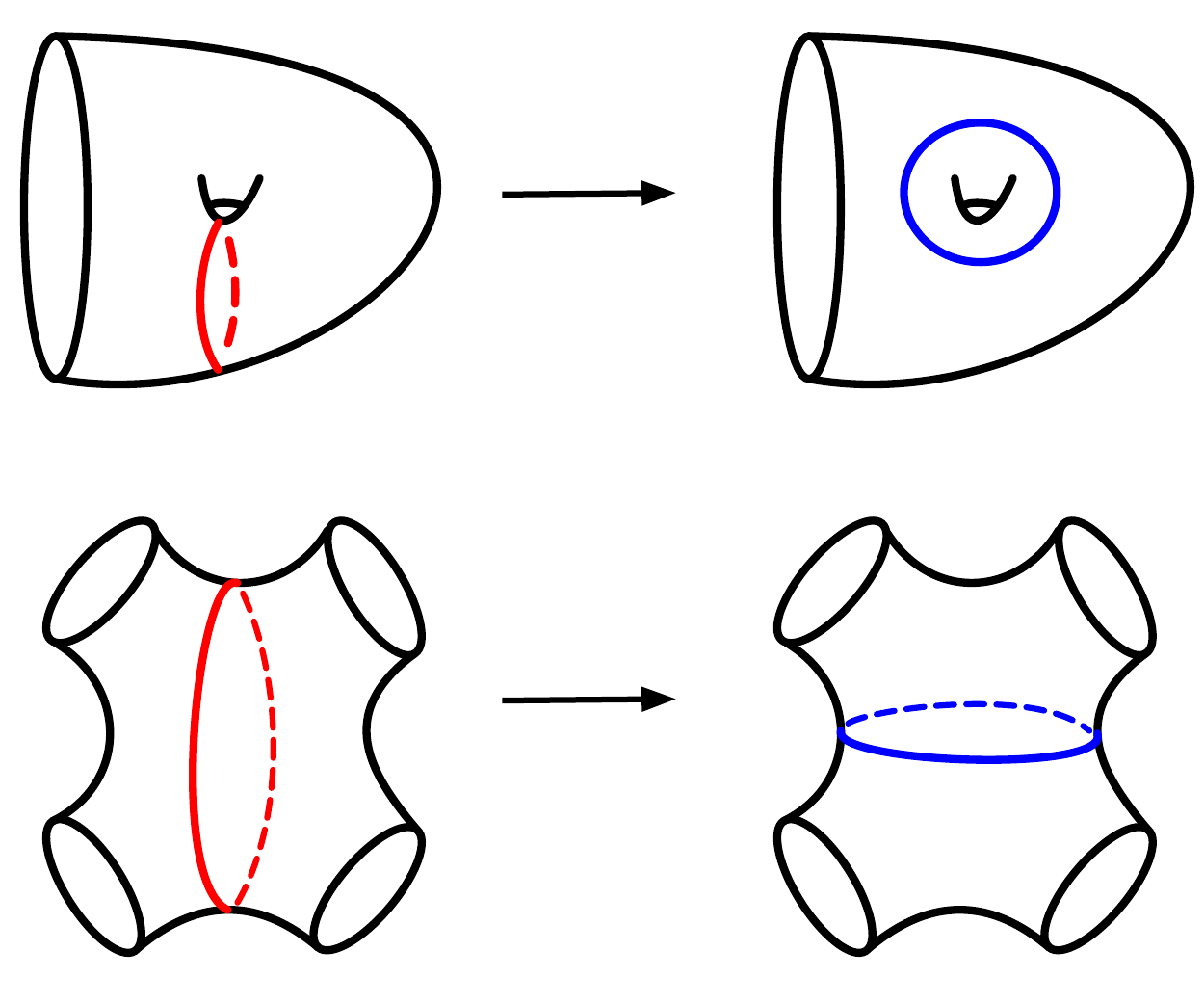}
    \caption{Two elementary moves}
    \label{pants-moves}
\end{figure}

\subsubsection{Hatcher-Thurston cut system}
Let $\Sigma$ be a closed orientable surface of genus $g$. 
Let  $C_1, \cdots, C_g$ be an unordered collection of $g$ loops in $\Sigma$,
whose complement $\Sigma -(C_1, \cdots, C_g)$ is a $2g$-punctured sphere.  An isotopy class of such
systems $\{C_1, \cdots, C_g\}$ is called a cut system ($C^{HT}(\Sigma)$).

Replacing some curve $C_i$ in a cut system by another simple loop  $C_i'$ that intersects $C_i$ transversely in one point and is disjoint from $C_j$ for all $j \neq i$, gives us another cut system.  Such a replacement is called a simple move.
We can define a cut system graph with vertices as cut-systems and an egde between them when two vertices are related by a simple move. This can be extended to a simplicial complex by attaching $3, 4,$ or $ 5$- gons to cycles of simple moves of the same order. See \cite{hatcher1980presentation} for more details.

\subsection{Distances in simplicial complexes}\label{Distances in simplicial complexes}

The 1-skeleton of simplicial complexes are graphs, that can be viewed as metric spaces with each edge being distance one. 
Distance between two vertices is given by the length of a shortest geodesic path joining them. 

We denote by $d(\cdot,\cdot)$ the distance in the curve graph between two curves. 
The distance formula for the curve complex obtained by Masur-Minsky in \cite{masur1999geometry}  is 
\[
d(x, y) \asymp \sum_{Y \subseteq S} \left[ d_{\mathcal{C}(Y)}(\pi_Y(x), \pi_Y(y)) \right]_K
\] 
Where:
\begin{itemize}
    \item \( x, y \) are markings or curves on the surface \( S \),
    \item \( Y \subseteq S \) ranges over essential subsurfaces,
    \item \( \pi_Y(x) \) is the projection of \( x \) to the curve complex \( \mathcal{C}(Y) \),
    \item \( d_{\mathcal{C}(Y)} \) is the distance in the curve complex of \( Y \),
    \item \( [\cdot]_K \) is a threshold function: 
    \[
    [a]_K = 
    \begin{cases}
    a & \text{if } a \geq K, \\
    0 & \text{otherwise},
    \end{cases}
    \]
    \item and \( \asymp \) means the sum is quasi-equal to the distance up to multiplicative and additive constants.
\end{itemize}

The \emph{dual distance} $\dpw (v, v')$ between two vertices in $\mathcal{C}^*(\Sigma)$ is the length of the shortest  
path in $\mathcal{C}^*(\Sigma)$ between them. 
We let $\dpants(v, v')$ be the distance between vertices $v$ and $v'$ in the pants complex.

\begin{remark}\label{dc and dp}
    Because of the one-to-one correspondence between the vertices of $\mathcal{C}^*(\Sigma)$ and the vertices of $\mathcal{C}_P(\Sigma)$,  
we can think of $v$ and $v'$ as being in either graph. An edge path in $\mathcal{C}_P$ maps to an edge path  
of the same length in $\mathcal{C}^*$, so  
\[
\dpw(v, v') \leq \dpants(v, v').
\]
\end{remark}

We denote by $d_{HT}(v, v')$  the shortest distance between vertices $v$ and $v'$ in the Hatcher-Thurston cut system complex.

\subsection{3-manifold decompositions along a surface}\label{3 mani}

There are various ways of decompositioning a 3-manifold along surface, the most well known being a handlebody decomposition or a Heegard splitting. This allows us to use structures on the splitting surface to study the 3-manifold.

\subsubsection{Heegaard Splittings}

A \emph{Heegaard splitting} of a closed, orientable 3-manifold \( M \) is a decomposition of \( M \) into two handlebodies \( H_1 \) and \( H_2 \) of the same genus, such that
\[
M = H_1 \cup_\Sigma H_2,
\]
where \( \Sigma = \partial H_1 = \partial H_2 \) is a closed, orientable surface called the \emph{Heegaard surface} \cite{hempel2004, jaco1980}. The genus of the Heegaard surface is known as the \emph{genus of the Heegaard splitting}. The minimal genus over all such splittings of \( M \) is called the \emph{Heegaard genus} of \( M \). 
Heegaard splittings play a central role in the classification of 3-manifolds and in understanding their topology via surface and group-theoretic methods \cite{schultens2014}.

\subsection{Distances and complexities in 3-manifolds}\label{dist 3 mani}

Let $H$ be a handlebody and let $\varphi\colon \Sigma \to \partial H$ be a homeomorphism.  
For a vertex $u \in C(\Sigma)$, write $u \in H$ if, for some loop $l$ in the isotopy class corresponding to $u$,  
$\varphi(l)$ bounds a disk in $H$. 
Let $(\Sigma, H_1, H_2)$ be a Heegaard splitting of a manifold $M$. Consider the inclusion maps  
\[
\Sigma \hookrightarrow H_i.
\]  
Each map suggests a set of vertices in $C(\Sigma)$ which are in $H_i$. The standard  
distance of $\Sigma$, as in Hempel \cite{MR1838999}, is  
\[
d(\Sigma) = \min\{d(u, u') \mid u \in H_1, u' \in H_2\}.
\]  
The distance $d(\Sigma)$ measures the irreducibility of $\Sigma$, in the sense that:
\begin{itemize}
    \item if $d(\Sigma) = 0$ then $\Sigma$ is reducible
    \item if $d(\Sigma) = 1$ then $\Sigma$ is weakly reducible
    \item if $d(\Sigma) = 2$ then $\Sigma$ has the disjoint curve property.
\end{itemize}

For a Heegaard splitting $(\Sigma, H_1, H_2)$, we say $v$ defines $H_i$ if $v$ is a pants decomposition of $\Sigma$ such that each curve of $v$ bounds a disk in $H_i$.  
The dual distance  
of $\Sigma$ is  
\[
D(\Sigma) = \min\{\dpw (v, v') \mid v \text{ defines } H_1,\, v' \text{ defines } H_2\}.
\]  
Note that $D(\Sigma) \geq d(\Sigma)$. 
Hempel has shown that there are genus two Heegaard splittings such that $d(\Sigma)$ is  
arbitrarily large. Thus, there are Heegaard splittings with $D(\Sigma)$ arbitrarily large.

Similarly we can also define the pants distance of a splitting surface,

\[
D^P(\Sigma) = \min\{\dpants (v, v') \mid v \text{ defines } H_1,\, v' \text{ defines } H_2\}.
\]  
For a Heegaard splitting $(\Sigma, H_1, H_2)$, we say $v$ defines $H_i$ if $v$ is a cut system of $\Sigma$ such that each curve of $v$ bounds a disk in $H_i$.  
The Hatcher-Thurston cut system distance  
of $\Sigma$ is 
\[
D^{HT}(\Sigma) = \min\{d_{HT} (v, v') \mid v \text{ defines } H_1,\, v' \text{ defines } H_2\}.
\]  

For handlebody knots, we extend the definitions above as follows: by a vertex $v$ defines $H_{i}$ we mean that each curve of $v$ bounds a disk or a punctured disk in $H_{i}$.

Johnson \cite{johnson2006heegaard} defines integral measures of complexity for Heegaard splittings based on the dual curve graph and on the pants complex, respectively, as below.

$$
A_g(\Sigma)=D\left(\Sigma_g\right)+b-g \text { and } A_g^P(\Sigma)=D^P\left(\Sigma_g\right)+b-g.
$$
Here \( b \) is the sum of the genera of the boundary components of the manifold.
 As the Heegaard splitting is stabilized, the sequence of complexities converges to a non-trivial limit depending only on the manifold, giving a stable invariant for the 3-manifold.

We similarly define and study $$A^{HT}_g(\Sigma)=D^{HT}\left(\Sigma_g\right)-g-n$$ in Section \ref{Hatcher-Thurston}.
Here $n$ is the number of $S^1\times S^2$ components of the prime decomposition of the manifold.

\section{Estimating distances in simplicial complexes and sub-complexes using curve complex}\label{section relation distance}

In this section, we obtain lower bounds on distances between vertices of various simplicial complexes in terms of distances of the components of the vertices in the curve complex.

\subsection{Distance relation between the curve graph and the dual curve graph}

In this section, we study the relation between the distance in the curve graph and that of the dual curve graph. The main results are Theorem \ref{prop dC and dP} and \ref{proposition general formula flipping}.
Proposition \ref{prop dC and dP} provides a lower bound of the distance in the dual curve complex in terms of distance in the curve complex between the curves and the number of curves. 
Proposition \ref{proposition general formula flipping} provides a geometric proof of the above result.  
Let $P$ be a pants decomposition. Let $N$ be the number of curves of $P$.

\begin{theorem}\label{prop dC and dP}
Let $P$ and $T$ be pants decompositions. 
For all $k\ge1$, if each curve of $P$ has a distance at least $k$ from any curve of $T$ then 
\[
\dpw (P,T) \geq k (N -1) + 1.
\]
\end{theorem}

\begin{proof} 
\textbf{Base case:} $k=1$. 
Suppose that for all $x_P\in P$ and for all $x_T\in T$,  
\[
d_C(x_P,x_T) \ge 1.
\]
In this case, curves of $P$ and $T$ are different. It takes at least $N$ steps to change from $P$ to $T$, hence
\[
\dpw (P,T) \geq N = 1\times(N-1)+1.
\]

\noindent
\textbf{Inductive step:} Assume that the statement of Proposition \ref{prop dC and dP} is true for some $k\in \mathbb{N}$. 
We will show that it is also true for $k+1$, that is,  
if for all $x_P\in P$ and for all $x_T\in T$, 
\[
d_C(x_P,x_T) \ge k+1
\]
then
\[
\dpw (P,T) \geq (k+1) (N -1) + 1. 
\]
Indeed, since for all $x_P\in P$ and for all $x_T\in T$,  
\[
d_C(x_P,x_T) \ge k+1 > k,
\]
by inductive hypothesis,
\[
\dpw (P,T) \geq k (N -1) + 1. 
\]
Let $Q$ be an arbitrary pants decomposition with 
\[
\dpw (P,Q) = k(N-1) 
< k (N -1) + 1.
\]
Using the contrapositive of the inductive hypothesis, 
there exists $x_P\in P$ and there exists $x_Q\in Q$ such that
\[
d_C(x_P,x_Q) < k.
\] 
Since every curve of $P$ is of distance at least $k+1$ apart for every curve of $T$, $x_Q$ cannot be a curve of $T$.
Similarly, for all $y_Q\in Q\setminus\{x_Q\}$, $y_Q$ cannot be a curve of $T$ either, because  
\[
d_C(x_P,y_Q)\le d_C(x_P,x_Q) + d_C(x_Q,y_Q) < k+1.
\]
Therefore, 
\[
\dpw (P,T) \geq 
k(N-1)+N
=
(k+1) (N -1) + 1. 
\]
By Principle of mathematical induction, we are done.
\end{proof} 

Furthermore, we also know how many times the curves of $P$ are \textit{flipped}. 
Here by \textit{flipping} a curve in a pants decomposition, we mean that it is replaced by another curve so that together with the other $N-1$ curves, they still form a pants decomposition.  
A curve is \textit{admissibly flipped $m$ times} if the resulting curve is of distance at least $m$ from it. 
From now on, we only consider admissible flips.  

\begin{lemma}\label{lemma k <= N+1}
Let $P$ and $T$ be pants decompositions. 
Let $k\in\mathbb{N}, 1\le k \le N+1$.
If for all $x_P\in P$ and for all $x_T\in T$,  
\[
d_C(x_P,x_T) \ge k
\]
then every curve of $P$ needs to flip at least $k-1$ times and 
there are at least $N-k+1$ curves such that each curve needs to flip at least $k$ times.
\end{lemma}

The contrapositive of the statement of Lemma \ref{lemma k <= N+1} is as follows.

\begin{quote} 
Let $k\in\mathbb{N}, 1\le k \le N+1$.
If there is a curve of $P$ that flips strictly less than $k-1$ times or there are strictly less than $N-k+1$ curves that each flip at least $k$ times, then there exist $x_P\in P$ and there exist $x_T\in T$ such that
\[
d_C (x_P, x_T) < k.
\] 
\end{quote}

\begin{proof}[Proof of Lemma \ref{lemma k <= N+1}]
\textbf{Base case:} $k=1$. 
Suppose that for all $x_P\in P$ and for all $x_T\in T$,  
\[
d_C(x_P,x_T) \ge 1.
\]
In this case, every curve of $P$ needs to flip at least once because otherwise, there is at least one curve that stays fixed, and the distance between it and itself is $0$. 

\smallskip 

\noindent
\textbf{Inductive step:} Assume that the statement of Lemma \ref{lemma k <= N+1} is true for some $k\in \mathbb{N}$, $1\le k \le N$. 
We will show that it is also true for $k+1$, that is,  

\begin{quote}
If for all $x_P\in P$ and for all $x_T\in T$,  
\[
d_C(x_P,x_T) \ge k+1
\]
then
every curve of $P$ needs to flip at least $k$ times to become a curve of $T$, and  among them 
there are at least $N-k$ curves of $P$ where each curve needs to flip at least $k+1$ times.    
\end{quote}

\noindent
By inductive hypothesis, every curve of $P$ needs to flip at least $k-1$ times and among them 
there are at least $N-k+1$ curves of $P$ such that each curve needs to flip $k$ times. 
Let $Q$ be \textit{any} pants decomposition obtained from $P$ 
by flipping $N-k$  curves each $k$ times, and flipping the $k$ other curves each $k-1$ times. 
Since there are strictly less than $N-k+1$ curves flipped at least $k$ times, using the contrapositive of the inductive hypothesis, 
there exists $x_P\in P$ and there exists $x_Q\in Q$ such that
\[
d_C(x_P,x_Q) < k.
\] 
Since every curve of $P$ is distance at least $k+1$ apart from every curve of $T$, $x_Q$ needs to flip one more time.
Similarly, for all $y_Q\in Q\setminus\{x_Q\}$, $y_Q$ needs to flip one more time, because  
\[
d_C(x_P,y_Q)\le d_C(x_P,x_Q) + d_C(x_Q,y_Q) < k+1.
\]
Therefore, every curve of $Q$ needs to flip one more time. Hence, every curve of $P$ needs to flip at least $k$ times and among them, at least $N-k$ curves each needs to flip $k+1$ times. 
\end{proof} 

In general, we have the following. 

\begin{proposition}\label{proposition general formula flipping}
Let $c\in \mathbb{N}, c\ge 1$. For all $k\in \mathbb{N}$ such that
\[
(c-1)N+2\le k \le cN+1,
\]
if each curve of $P$ has a distance at least $k$ from any curve of $T$ then 
every curve of $P$ flips at least $k-c$ times and at least $cN+1-k$ curves flips at least $k-c+1$ times.  
\end{proposition} 
\noindent
Note that the condition $(c-1)N+2\le k \le cN+1$ guarantees that for any $c\ge1$, we always have $k\ge c$.

\begin{proof}
\textbf{Base case ($c=1$):} We need to show that
\begin{quote}
    For all $k\in \mathbb{N}$ such that
    \[
    2\le k \le N+1,
    \]
    if each curve of $P$ has a distance at least $k$ from any curve of $T$ then 
    every curve of $P$ flips at least $k-1$ times and at least $N+1-k$ curves flip at least $k$ times.   
\end{quote}
This is proved in Lemma \ref{lemma k <= N+1}.

\medskip

\noindent
\textbf{Inductive step:}
Assume that the statement of Proposition \ref{proposition general formula flipping} is true for some $c\in \mathbb{N}$, $c\ge 1$. 
We will show that it is also true for $c+1$, that is  

\begin{quote}
For all $k\in \mathbb{N}$ such that
\[
cN+2\le k \le (c+1)N+1,
\]
if each curve of $P$ has a distance at least $k$ from any curve of $T$ then 
every curve of $P$ flips at least $k-c-1$ times and at least $(c+1)N+1-k$ curves flip at least $k-c$ times.  
\end{quote}
This can be proved similarly as before by induction on $k$. 
We are done.  
\end{proof}

\begin{remark}
    Note that previously we only knew that $D(\Sigma)\ge d(\Sigma)$.
    In terms of $d(\Sigma)$ and $D(\Sigma)$ for a Heegaard splitting $(\Sigma, H_1, H_2)$, Theorem \ref{prop dC and dP} gives a better comparison, as it can be seen as
    \[
    D(\Sigma) \geq (N -1)d(\Sigma) + 1.
    \] 
    Indeed, since  
    \[
d(\Sigma) = \min\{d(u, u') \mid u \in H_1, u' \in H_2\}
\]  
and 
    \[
D(\Sigma) = \min\{\dpw (v, v') \mid v \text{ defines } H_1,\, v' \text{ defines } H_2\}.
\]  
Let $P, T$ be pants decompositions define $H_1, H_2$ such that 
$D(\Sigma) = \dpw (P,T)$.
Let
\[
k:=\min\{d(u, u') \mid u \in P, u' \in T\}
\ge d(\Sigma).
\]
Then
\[
D(\Sigma) = 
\dpw (P,T)\ge k (N -1) + 1 
\ge  (N -1) d(\Sigma) + 1.
\]
\end{remark}
 
\subsection{Distance relations between general simplicial complexes and curves complexes}\label{general complexes}
 The above bounds obtained on distances in the dual curve complex in terms of distances in curve complexes can be extended to other complexes such as the pants complex, Hatcher-Thurston cut system, etc. In particular, we can do it for any simplicial complex with which the following criterion is satisfied. We call such complexes admissible multi-curve complex.

\begin{enumerate}
  \item For some fixed $N$, each vertex is a collection of $N$ pairwise disjoint simple essential closed curves on the surface.
      \item Two vertices are connected by an edge if they differ by an \textit{admissible move}.
  \end{enumerate}

 \begin{definition}
   An admissible move consists of replacing one simple curve of a vertex by another curve (homotopically distinct) satisfying some conditions that include that the new curve must intersects the old one, is disjoint from the other existing curves, and after replacing we obtain a new vertex of the complex.
    
\end{definition}

 Note: The set of admissible moves can have other restrictions and can be a small set, like only 2 in case of pants graph. 

\begin{theorem}\label{general complex distance}
     Let $V_1$ and $V_2$ be two vertices of an admissible multi-curve complex.
 For all $k\ge1$, if each loop of $V_1$ has a distance at least $k$ from any loop of $V_2$ then 
 \[
 d_{cw} (V_1,V_2) \geq k (N -1) + 1.
\]

  where $d_{cw}$ is the distance in the 1-skeleton of the simplicial complex. 
 \end{theorem}

 \begin{proof}
     The proof can be done via induction in an almost identical fashion as \ref{prop dC and dP}.
     Note that in the proof of \ref{prop dC and dP} we only use the fact that the image of a curve under an admissible move is disjoint from the existing curves and a move creates a new curve that's homotopically distinct.
 \end{proof}

\subsection{Relation between pants distances of filling pairs and curve complex}

We can impose additional restrictions on our vertex set and study the relations of specific subsets to the curve complex. In particular, we can ask if the lower bound can be improved in these cases. We also can ask when the vertices satisfy additional criteria, what conditions does this impose on our 3-manifolds in the settings of section \ref{Applications to 3-manifolds}.
In this section, we explore the effect of an additional condition on the vertices in the dual curve complex. We require the two vertices whose distance we want to compute to fill the surface. In this case, we get the following lower bound.

\begin{proposition} 
Let $S$ be a closed surface of genus $g\ge2$. 
    Let $P$ and $P'$ be two pants decompositions of $S$ and they fill $S$. Then
    \[
    \dpw(P,P')\ge 3g-3.
    \]
    Furthermore, 
given any closed surface of genus $g\ge2$, we can always find two pants decompositions on it that are distance $3g - 3$ apart. 
\end{proposition}
\begin{proof}
Since pants decompositions contain exactly $3g-3$ curves, via the pigeon-hole principle there is at least one of the original curves that is left unchanged, which means this is part of both pants decompositions. This means no other curve in either pants decomposition intersects it, which proves our collection is not filling. Thus, we get a contradiction.

We now show that $ 3g-3$ is the minimum possible distance: given any closed surface of genus $g\ge2$, we can always find two pants decompositions that fill it and are distance $3g - 3$ apart. 
 We begin by demonstrating this on the surface of genus $2$ and $3$, as in Figures \ref{figg2} and \ref{figg3}.

\begin{figure}[h]
    \centering \includegraphics[width=0.70\linewidth]{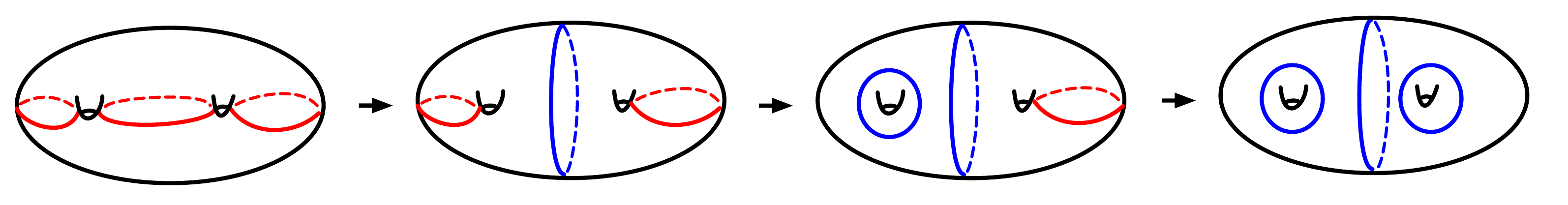}
    \caption{Distance $3$ on genus $2$ surface}
    \label{figg2}
\end{figure}

\begin{figure}[h]
    \centering \includegraphics[width=\linewidth]{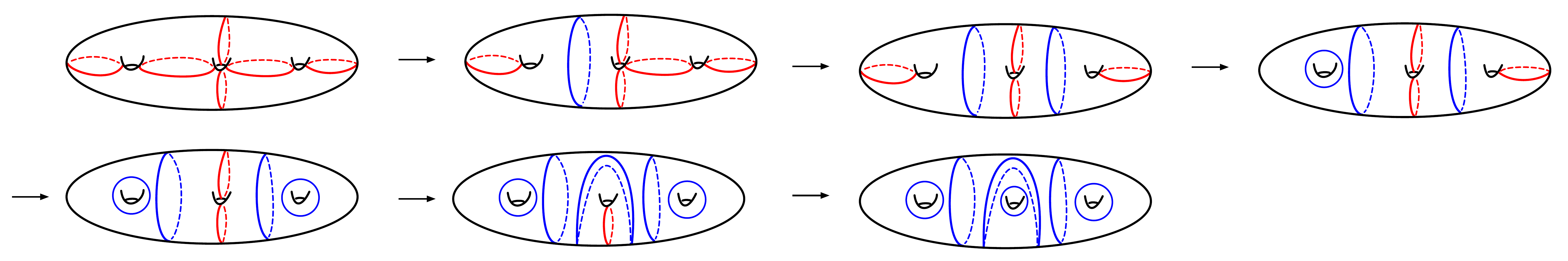}
    \caption{Distance $6$ on genus $3$ surface}
    \label{figg3}
\end{figure}

To extend our proof to other closed surfaces of genus $g\ge4$ the idea is to start with a pants decomposition that consists of a curve going between each pair of consecutive genera and two more around each genus (vertical ones) except for the genus on the two ends. 
Now for the moves, we start with the 
moves that take each (horizontal) curve between consecutive genera to a separating vertical one, 
then two moves on both ends, 
and then in each $X$-piece around the middle genera we do two more moves as in the genus 3 case. It is clear these are pants decompositions and two pants decompositions fill $S$. We have performed a move on each curve we began with so the distance is $3g -3$.
\end{proof}

\begin{remark}
Note that since
\[
\dpw(v, v') \leq \dpants(v, v'),
\]
we also have
\[
\dpants(P, P')\ge 3g-3.
\]
for any two pants decompositions $P$ and $P'$ that fill $S$. 
\end{remark}

\begin{lemma}

Given a pants decomposition $P$ on any closed surface, we can construct another pants decomposition  $P'$ such that $P$ and $P'$ fill and $d(P, P') = 3g -3 $.
\end{lemma}
\begin{proof}

The pants decomposition $P'$ can be constructed as follows. We first get the dual graph of $P$: each vertex corresponds to a pants, and each vertex gives one tripod (see Figure \ref{fig:from P construct P'}). From this dual graph of $P$, we can form a set of disjoint simple close curves on $S$ by taking an $\varepsilon$-neighborhood of the dual graph. Then extend this set of curves to a pant decomposition which we choose it to be $P'$.

\begin{figure}
    \centering
    \includegraphics[width=0.35\linewidth]{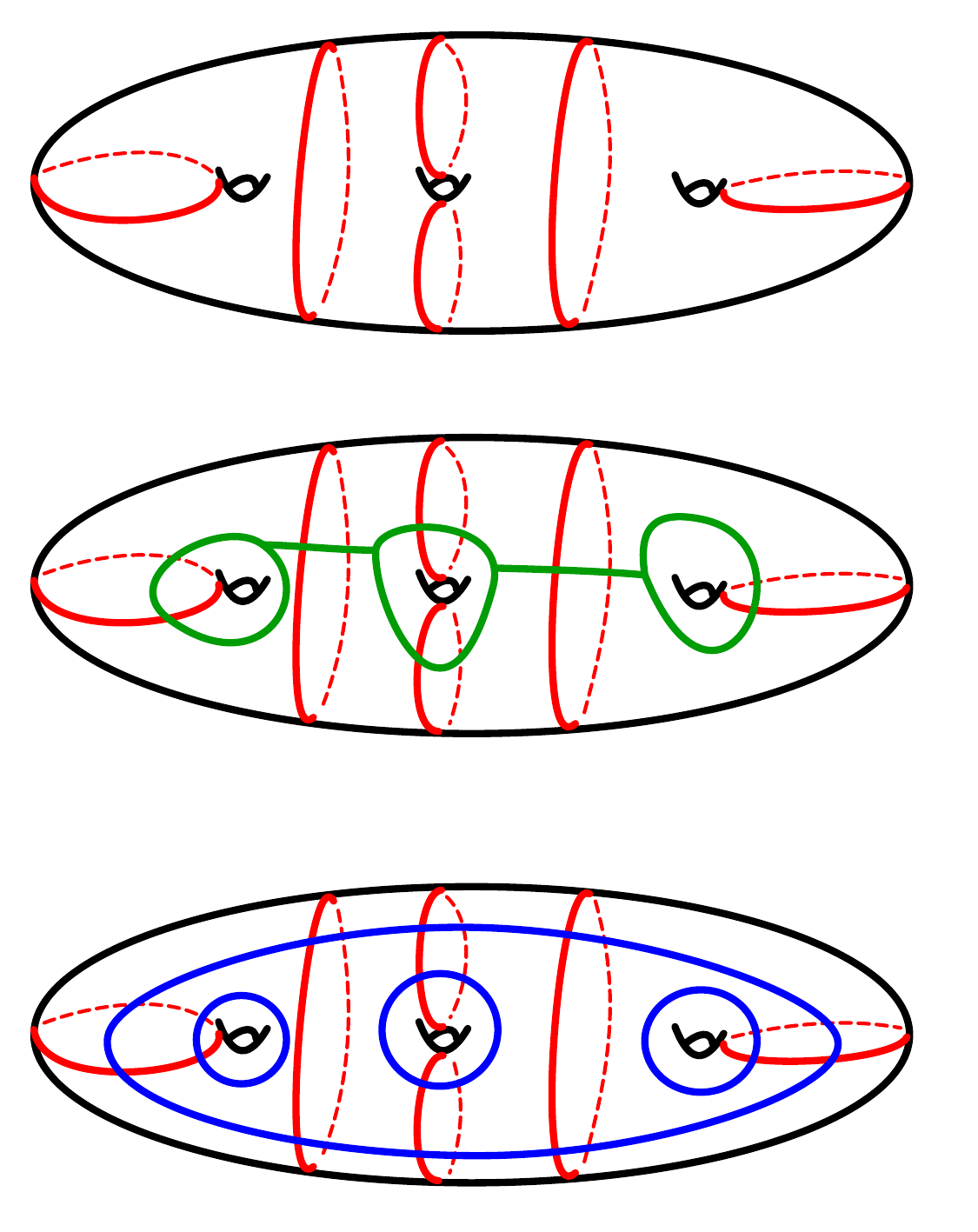}
    \caption{Given $P$, constructing $P'$ such that $P$ and $P'$ fill $S$.}
    \label{fig:from P construct P'}
\end{figure}
\end{proof}

\begin{remark}
    Note, this is the best possible lower bound since it is realized. Furthermore, this corresponds to $k=1$ in \ref{prop dC and dP}, so any pair $P, P'$ that realizes the lower bound, for example as in Remark 4.3, must have $k = 1$.
\end{remark}

\section{Applications to 3-manifolds}\label{Applications to 3-manifolds}

Hempel first introduced the idea of using curve complex to study complexity of 3-manifolds via Heegard splitting. Over the last few decades this approach of using 2-dimensional complexes associated with the splitting surface to study complexities has been extended to include other complexes like the dual curve complexity and pants graph complexity $
A_g(\Sigma)=D\left(\Sigma_g\right)+b-g \text { and } A_g^P(\Sigma)=D^P\left(\Sigma_g\right)+b-g
$ defined using dual curve complex and pants complex, respectively and obtaining stable invariants for 3 manifolds.
Here \( b \) is the sum of the genera of the boundary components of the manifold.

Following this approach,  we use the Hatcher-Thurston cut system to define a complexity for 3-manifold and the dual curve complex to define a complexity for handlebody-knots. These are done along the lines of similar invariants defined by Jesse Johnson using pants and curve complexes.

More generally, we remark that this kind of complexity may be defined using other simplicial complexes; however, we need an additional condition that $k \neq 0$ to get lower bounds using results from Section \ref{section relation distance}. 
In the following section, we provide bounds independent of $k$.

\subsection{Heegaard splittings and Hatcher-Thurston cut-systems}\label{Hatcher-Thurston} 

We consider a closed 3-manifold and a genus $g$ Heegard splitting.
In this case, each vertex is a cut system. Two vertices are connected by an edge if the two corresponding systems differ by one curve, where the old curve intersects the new curve once. 

Let $n$ be the number of $S^1\times S^2$ components of the prime decomposition of $M$. Let $\Sigma = \Sigma_g$ be a splitting surface of genus $g$. Note that $n$ is well-defined and finite. Let $A^{HT}_g(\Sigma) = D^{HT}(\Sigma_g)-g-n$ be the Hatcher-Thurston cut system complexity. 
\begin{theorem}
    The limit $\lim_{g\rightarrow\infty} A^{HT}_g(\Sigma)$ exists and is a 3-manifold invariant.
\end{theorem}

The proof is similar to Johnson's proof of Theorem 17 in \cite{johnson2006heegaard}, for the dual curve complex and follows from the following sequence of lemmas. Let's assume $n=0$ for simplicity.

\begin{lemma}
 $$D^{HT}(\Sigma) \geq g.$$

\end{lemma}

\begin{proof}
    The proof is identical to the proof for dual curve distance in Lemma 8 in \cite{johnson2006heegaard}. The proof only uses a sub-collection of $g$ curves to obtain a contradiction, as in the Hatcher-Thurston cut system each vertex consists of $g$ curves satisfying the same condition, we get the required bound.
\end{proof}

\begin{lemma}
    The sequence $A^{HT}_g(\Sigma)$ is non-increasing.
\end{lemma}

\begin{proof}
    There are $g$ curves in a cut system. For each stabilization, we increase the genus by one, and we can choose a new curve to be the red curve as shown in Figure \ref{staba}. Swapping the red curve to the blue curve bounds a disk in the other handlebody. This implies that $D^{HT}(\Sigma_{g+1})\leq D^{HT}(\Sigma_{g}+1).$ Therefore, $A^{HT}_{g+1}(\Sigma) = D^{HT}(\Sigma_{g+1})-(g+1)\leq D^{HT}(\Sigma_{g})+1-(g+1)=A^{HT}_g(\Sigma).$
\end{proof}

\begin{figure}[ht!]
\centering
\includegraphics[width=0.4\linewidth]{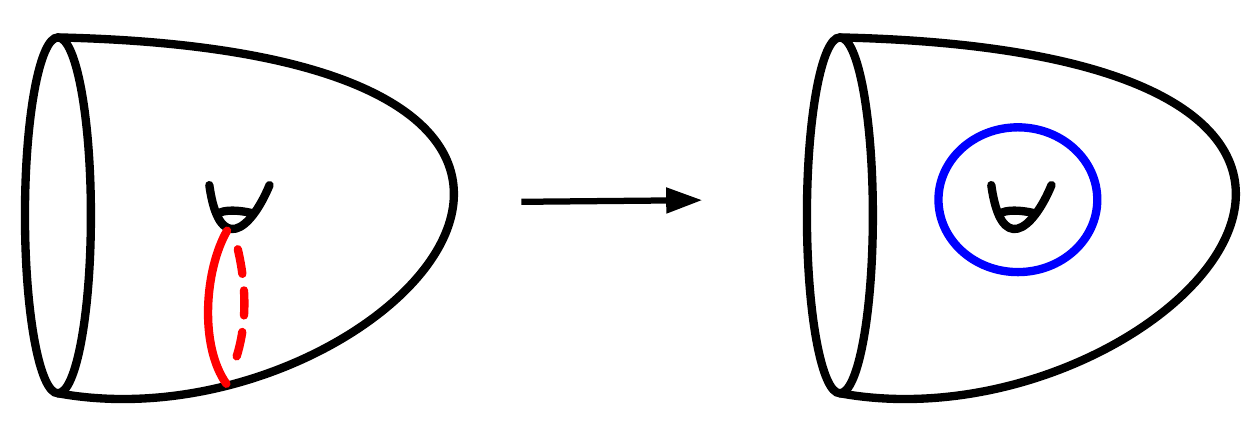}\label{staba}
\caption{A simple move}
\end{figure}

\begin{figure}[ht!]
\centering
\includegraphics[width=0.9\linewidth]{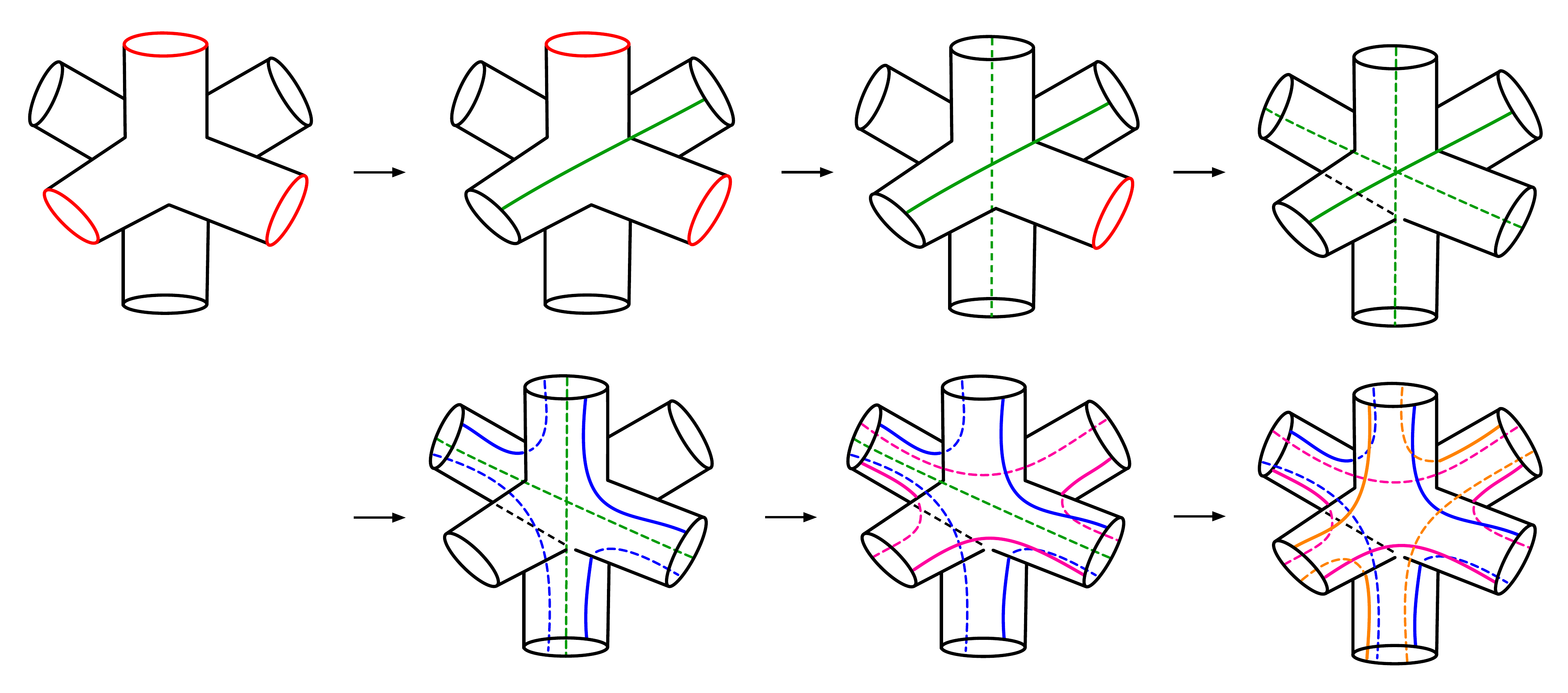}\label{3torus}
\caption{The Hatcher-Thurston distance of the 3-torus is 6}
\end{figure}

\begin{example}
It is well known that the Heegaard genus of the 3‑torus $T^3$ is three.  In fact, Boileau and Otal proved that the Heegaard splitting of any genus $g\geq 3$ is unique \cite{boileau1990scindements}.  That is, it is obtained by stabilizing the minimal‑genus splitting.  The Hatcher–Thurston distance $D^{HT}(\Sigma_3)$ of the 3‑torus is $6$.  Figure~\ref{3torus} shows a path of length~$6$, giving the upper bound.  To see the matching lower bound, observe that each curve must move at least once.  If not, there would be an $S^1\times S^2$ summand, contradicting the fact that the 3‑torus is prime.  In fact, each curve must move at least twice.  To see this, suppose some curve moves only once.  Then there exist curves $\alpha$ bounding a disk in $H_1$ and $\beta$ bounding a disk in $H_2$ such that $|\alpha\cap\beta|=1$.  This implies that the Heegaard splitting is stabilized, so the 3‑torus would admit a splitting of genus strictly less than~$3$, which is a contradiction.  Therefore,
\[
A_3(\Sigma) \;=\; 6 - 3 \;=\; 3.
\]

Now, for a genus $g>3$ Heegaard splitting of $T^3$, the same reasoning shows that
\[
A_g(\Sigma) \;=\; g + 3 - g \;=\; 3.
\]
Indeed, consider a cut system with $g$ curves.  Each curve must move at least once because $T^3$ has no $S^1\times S^2$ summand.  If all curves moved only once, we could sequentially destabilize and obtain a 3‑manifold homeomorphic to $S^3$.  Thus, at least three curves must move twice.

\end{example}

\subsection{Applications to handlebody-knots in 3-manifolds}\label{section handle body links}

Ozawa in \cite[Theorem 3.2]{ozawa2021stable} proved a version of stable equivalence for bridge positions of handlebody-knots. In this section, we define a complexity for handlebody-knots in the spirit of Hempel \cite{MR1838999} and Johnson\cite{johnson2006heegaard} and using it an invariant for handlebody-knots. We show that it converges under the stabilization moves in \cite{ozawa2021stable}.

\begin{figure}[h]
    \centering
    \begin{tikzpicture}
        \node[anchor=south west, inner sep=0] (image) at (0,0) {\includegraphics[width=0.6\linewidth]{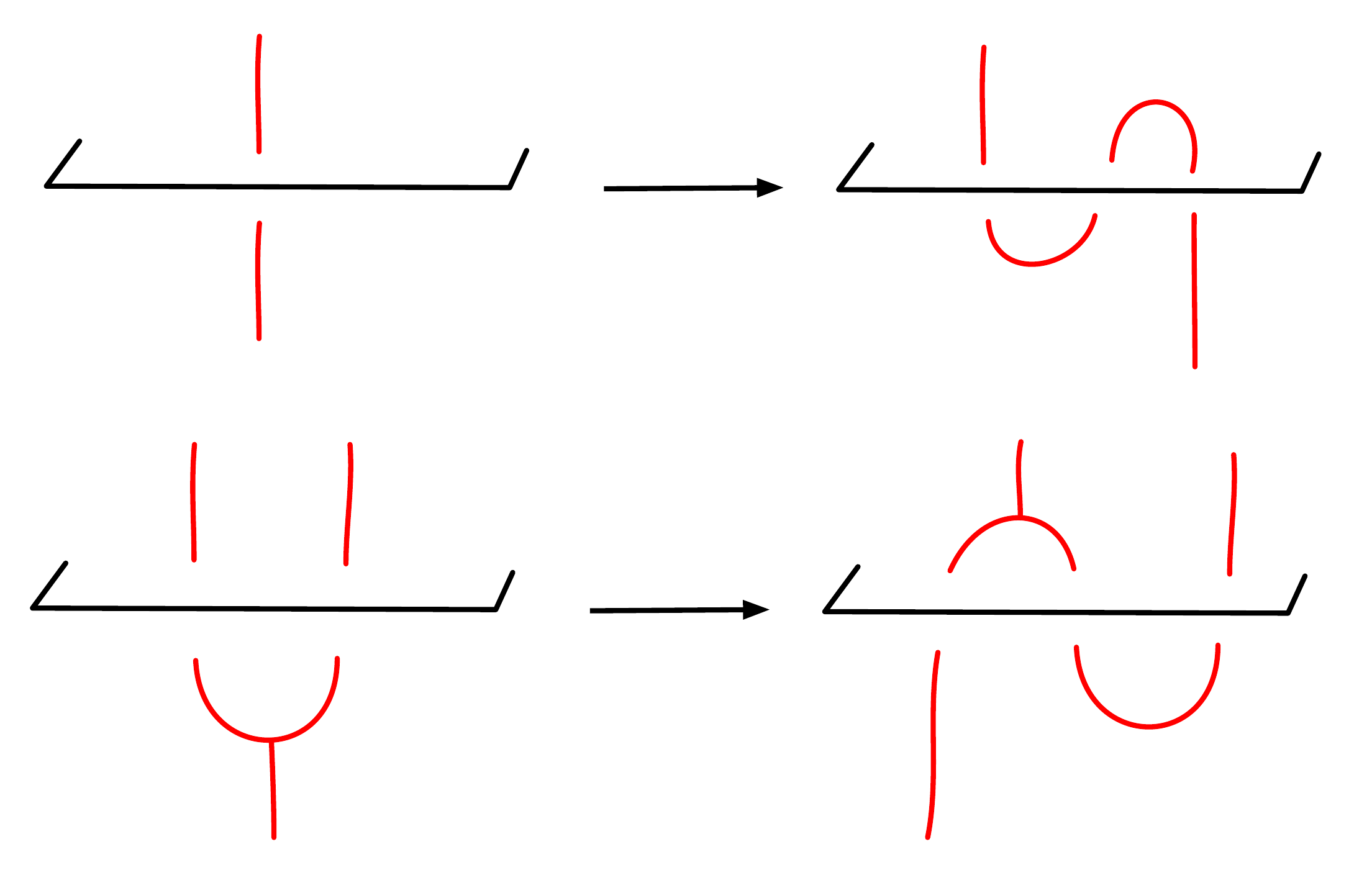}};
        \begin{scope}[x={(image.south east)}, y={(image.north west)}]
             \node at (0.50,0.84) {{\small Move $S_1$}};
            \node at (0.50,0.37) {{\small Move $S_2$}};
           
        \end{scope}
    \end{tikzpicture}
    \caption{Two moves on the spine of the handlebody-knot}
    \label{fig:2moves}
\end{figure}

\subsubsection{Definition and properties of handlebody-knots} 
We review the definition and some properties of handlebody-knots.
\begin{definition}
    A \textit{handlebody-knot} \textit{HK} is an embedding of finitely many handlebodies of positive genus in the 3-sphere $S^3$. Two handlebody-knots are equivalent if they are ambient isotopic.
\end{definition}
It is well-known that one does not lose information by considering a spine of the handlebody-knot $HK$: a graph whose regular neighborhood is $HK$. This is advantageous because 3-dimensional objects are now captured by diagrams of 1-dimensional objects.
\begin{definition}
    A handlebody-knot $HK$ is \textit{split} if there
exists a 2-sphere $S$ in $S^3$ such that $S \cap HL = \emptyset$ and both components of the complement $S^3
\backslash S$ have non-trivial intersection with $HK$.
\end{definition}
\begin{definition}
   An \textit{$n$-decomposing sphere} $S$ for a handlebody-knot $HK$ in $S^3$ is a 2-sphere that intersects $HK$
at $n$ disjoint disks so that the $n$-punctured sphere $S \cap (S^3 \backslash HL)$ is incompressible and non-boundary parallel.
\end{definition}
\begin{definition}
   A handlebody-knot is \textit{$n$-decomposable} if it admits an $n$-decomposing sphere.
\end{definition}

\begin{remark}
    Observe that if the bridge sphere is perturbed, we can find a natural 2-sphere $S$ that intersects a handlebody-knot in 2 disjoint disks (see Figure \ref{reducingsphere}). However, this 2-sphere is not a 2-decomposing sphere since the corresponding punctured sphere $S \cap (S^3 \backslash HL)$ is in fact boundary parallel.  
    We list three forbidden 0,1,2-decomposing spheres in Figure \ref{fig: forbidden decomposing sphere}.

    \begin{figure}
        \centering
        \includegraphics[width=0.5\linewidth]{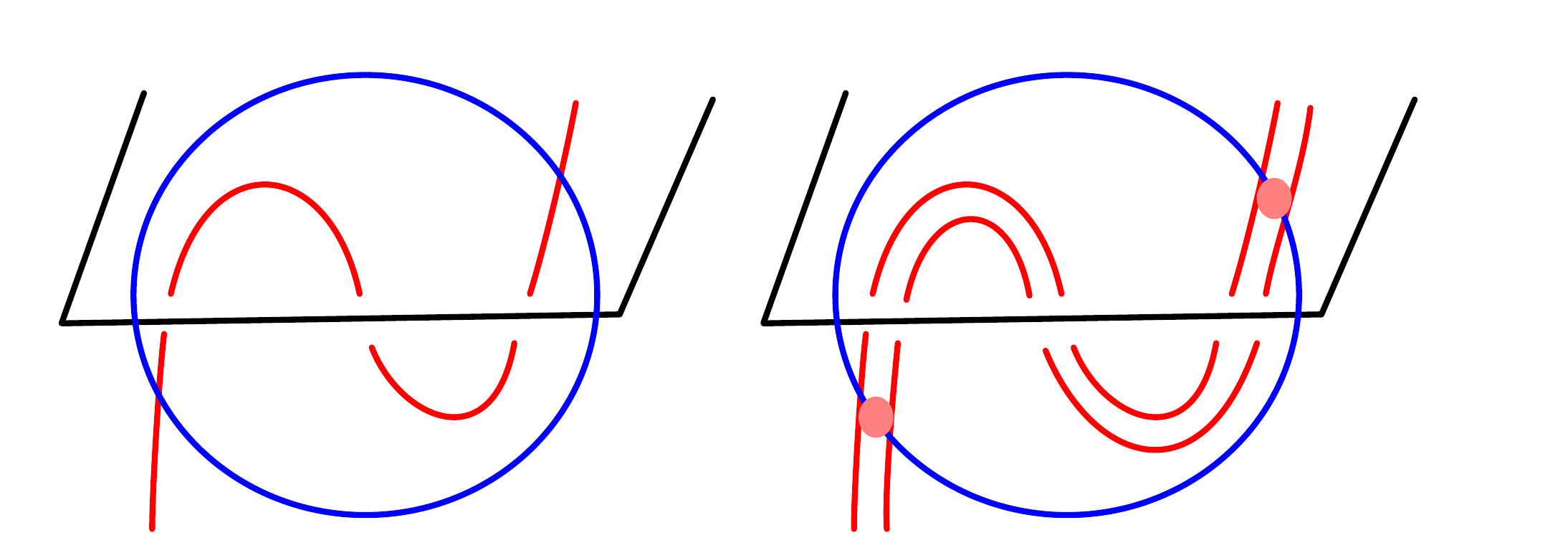}
        \caption{An admissible decomposing sphere that intersects the handlebody-knot in two parallel disks.}
        \label{reducingsphere}
    \end{figure}

    \begin{figure}
        \centering
        \includegraphics[width=0.75\linewidth]{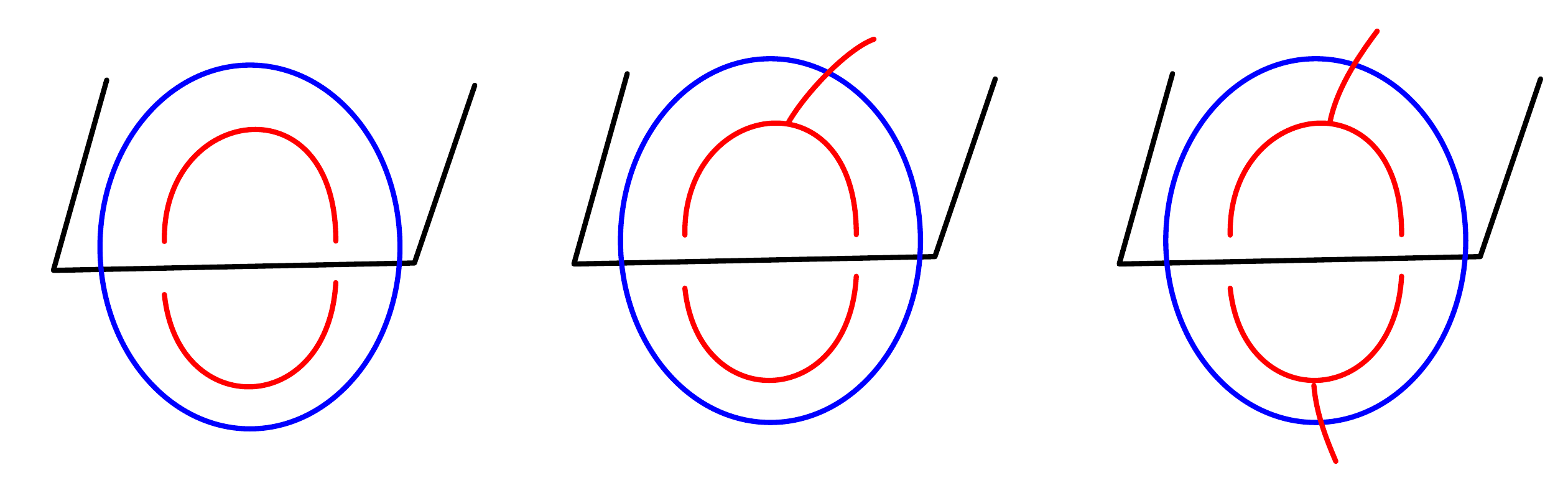}
        \caption{Forbidden $0,1,2$-decomposing spheres (left to right, respectively)}
        \label{fig: forbidden decomposing sphere}
    \end{figure}
\end{remark}

\subsubsection{Axiomatic way to define an invariant of knotted handlebodies in a 3-manifold}
Hempel \cite{MR1838999} was the person who originally formulated a complexity measure for knotted objects in 3-manifolds of this style using the curve complex. Johnson \cite{johnson2006heegaard} later on generalized to the pants complex. There are numerous other works, but they all follow this same style. Let $(M,\Lambda)$ be a (3-manifold, handlebody-knot) pair.

\textbf{Step 1}: Decompose $(M,\Lambda)$ as a Heegaard splitting (also called a bridge splitting). This is a decomposition of $(M,\Lambda)=(H_1,T_1)\cup_{\Sigma} (H_2,T_2)$, where $(H_j,T_j)$ is a trivial tangle in a handlebody and the common intersection $\Sigma$ is a surface possibly with marked points.

\textbf{Step 2}: Pick a favorite simplicial complex $\mathcal{C}(\Sigma),$ where each vertex is a collection of essential pairwise nonisotopic simple closed curves on $\Sigma$. Two vertices differ by an edge if the corresponding collections differ by an elementary move.

\textbf{Step 3}: Let $\mathcal{D}_j$ be the set of vertices where each of the corresponding curves bounds a disk or a once-punctured disk in $(H_j,T_j)$. The complexity for of the splitting $(H_1,T_1)\cup_{\Sigma} (H_2,T_2)$ is the distance using the edge metric $d(\mathcal{D}_1,\mathcal{D}_2)$ measured in our choice of the complex of curves on $\Sigma.$

If one wants an invariant out of this construction, there is an extra step to perform. Since Hempel distance is more well-known, we discuss complexes where each vertex contains more than one curve. 

\textbf{Step 4a}: Show that as one stabilizes the splittings, the complexity converges to a well-defined number. This is more in the spirit of \cite{johnson2006heegaard}.

Step 4a above can be difficult to carry out since one would have to calculate the complexity for infinitely many splittings. To make the task easier, some authors chose another route to define invariants.

\textbf{Step 4b}: Minimize the complexity in Step 3 over all minimal Heegaard splittings.

\subsubsection{Handlebody-knot invariant from the dual curve complex}
Let $D(\Sigma_c)$ denote the dual curve distance of a bridge sphere, that is the minimum of the dual curve graph distance between pants decompositions that define the two handlebodies corresponding to a Heegaard slitting along a $c$-punctured sphere, respectively. 
In handlebody-knot theory, there are several ways to decompose a handlebody-knot into simpler handlebody-knots with smaller bridge indices. In this work, we focus on order 1 and order 2 connected sum. 
If a handlebody-knot admits a bridge splitting with 3 punctures, then the handlebody-knot is trivial. 
Our handlebody-knots are not $0,1,2$-decomposable, meaning that there does not exist a decomposing sphere that intersects the handlebody knot at 0, 1, or 2 disks, such that these disks are not parallel. In the literature, this condition implies that the handlebody-knots are irreducible and indecomposable.

Let \( s_i \ (i = 1, 2) \) denote the number of \( S_i \) moves that have been performed on \( \Sigma_c \), where \( c \) represents the number of initial punctures.  
After performing these moves, the number of punctures becomes \( c + 2s_1 + s_2 \).  
We shall denote the resulting surface by \( \Sigma_{c, s_1, s_2} \) or \( \Sigma_{c + 2s_1 + s_2} \). 
Define 
\[
B_\Sigma(c,s_1,s_2) = \frac{1}{4}D(\Sigma_{c,s_1,s_2}) - \frac{s_1+s_2}{4}.
\]

The main result of this section is Theorem \ref{main invariant}.  

\begin{lemma}\label{bound below invariant}
For an irreducible and indecomposible handlebody-knots,
\[D(\Sigma_n) \ge \frac{n}{4}.\]
where $n\ge4$ is the number of punctures. 
\end{lemma}

\begin{proof} 
Denote by $P_1$ and $P_2$ the above and below pants decompositions of the Heegaard surface $\Sigma_n$. 
We will prove by induction in $n$.

\paragraph{Base case:}
For $n=4$, we need to show that 
\[D(\Sigma_4) \ge \frac{4}{4}=1.\] 
If $D(\Sigma_4) =0 < 1$, then $P_1$ and $P_2$ have a common pants curve $\gamma$. Since $\gamma$ is the boundary of either a disk or a punctured disk, it gives rise to a $0,1,2$-decomposing sphere, and all these decomposing sphere are forbidden (see Figure \ref{fig:decomposable}). This contradicts our hypothesis that our handlebody-knots are irreducible and indecomposible. 
\begin{figure}[h]
    \centering
    \includegraphics[width=0.5\linewidth]{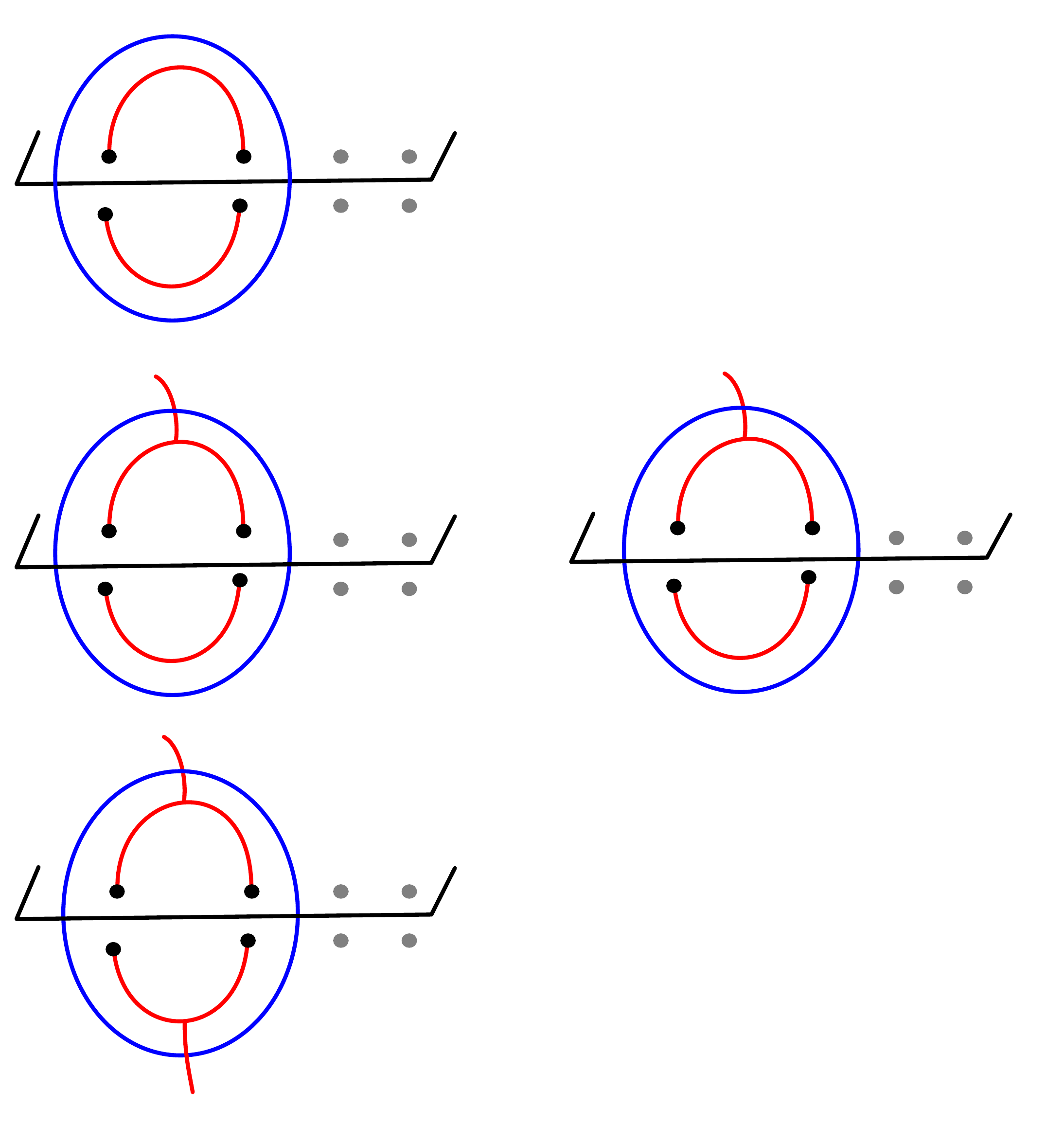}
    \caption{A curve bounds two punctures that is fixed along the path gives rise to a forbidden 0,1,2-decomposing spheres}
    \label{fig:decomposable}
\end{figure}

Due to our inductive argument as in what follows, we treat the case $n=5$ separately.
For $n=5$, each pants decomposition  $P_1$ and $P_2$ has two pants curves. Assume $P_1=\{\gamma_1,\gamma_2\}$. Each of them must be the boundary component of a pair of pants whose the other two components are two punctures among the $n=5$ punctures of $\Sigma_5$. 
By similar argument as in Base case, both $\gamma_1,\gamma_2$ do not belong to $P_2$ because otherwise it leads to a contradiction to our hypothesis that our handlebody-knots are irreducible and indecomposible. 

\paragraph{Inductive step:}
Suppose the claim holds for all $i\in[4,n-1]\cap\mathbb{N}$, meaning that 
\[
D(\Sigma_i) \ge \frac{i}{4}
\]
for all $i\in[4,n-1]\cap\mathbb{N}$. 
Consider the new Heegaard surface that arises after performing either move of type 1 (an $S_1$) or of type 2 (an $S_2$)  
on the handlebody-knot.
The current Heegaard surface is an $(n-1)$-punctured sphere. Depending on the type of the performed move, the new Heegaard surface is a sphere with either $n$ or $n+1$ punctures: 
$\Sigma_{n}$ arises when performing a move of type 2  and $\Sigma_{n+1}$ arises when performing a move of type 1. 
The number of curves in a pants decomposition is $n-3$ for $\Sigma_{n}$ and $n-2$ for $\Sigma_{n+1}$. We need to show that
\[
D(\Sigma_n) \ge \frac{n}{4}
\]
and 
\[
D(\Sigma_{n+1}) \ge \frac{n+1}{4}.
\]
For simplicity, in what follows we only write down the proof for the lower bound on $D(\Sigma_n)$. It is similar argument for $D(\Sigma_{n+1})$. 
We will prove by contradiction. Assume that
\[
D(\Sigma_n) < \frac{n}{4}
\]
We first notice that there must be a curve $\gamma$ that is fixed along the path from $P_1$ to $P_2$. Indeed, if the curves of $P_1$ are all different from the curves of $P_2$, the dual curve distance $\dpw(P_1,P_2)$ must be at least the number of pants curves, which is $n-3$ curves for $\Sigma_n$, hence
\[D(\Sigma_n)\ge n-3 \ge \frac{n}{4},
\] 
which is not the case. 
The curve $\gamma$ separates the $n$ punctures in $\Sigma_n$ into two sets of punctures.  
Let $n'$ and $n''$ be the numbers of punctures. 
Since $\gamma$ is a pants curve, we know that $n',n''\ge 2$.  
If either $n'=2$ or $n''=2$, then similar to the base case, $\gamma$ gives rise to a forbidden $0,1,2$-decomposing sphere, contradicting our hypothesis that our handlebody-knots are not $0,1,2$-decomposable (see Figure \ref{fig:decomposable}). 
Therefore,
\[
n',n''\ge 3.
\]
\paragraph{Case 1: $n'=n''=3$.} In this case $n=n'+n''=6$.  
Since $\gamma$ is a pants curve, there must exist a pants curve $\gamma'$ of $P_1$ bounds two punctures among the $n'=3$ punctures, and another pants curve $\gamma''$ of $P_1$ bounds two punctures among the $n''=3$ punctures. 
If $\gamma'$ stays fixed along the path from $P_1$ to $P_2$, using similar argument as in the base case, $\gamma'$ gives rise to a forbidden $0,1,2$-decomposing sphere, contradicting our hypothesis that our handlebody-knots are not $0,1,2$-decomposable (see Figure \ref{fig:decomposable}). 
Therefore $\gamma'\notin P_2$.
Similarly, $\gamma''\notin P_2$.
Then,
\[
D(\Sigma_n)\ge 2 \ge \frac{n}{4}
\] 
\paragraph{Case 2:}
If $n'=3$, we have
$n''= n-3\ge4$.  
By inductive hypothesis and using similar argument as in Case 1, we have
\[
D(\Sigma) \ge 1+ \frac{n''}{4} \ge \frac{n}{4}.
\]  
\paragraph{Case 3:} Both $n',n''\ge4$.
By inductive hypothesis and using similar argument as in Case 1, we have
\[
D(\Sigma_{n})
\ge 
\frac{n'}{4}+ \frac{n''}{4}
=\frac{n}{4}.
\]
By the principle of strong induction, we are done.  
\end{proof} 

In summary, we start from \( \Sigma_c \)  where \( c \) is the initial number of punctures.  
After performing \( s_1 \) type 1 moves and \( s_2 \) type 2 moves, the number of punctures becomes \( c + 2s_1 + s_2 \).  
By Lemma \ref{bound below invariant}, for all \( s_1, s_2 \), we have
 \[
 D(\Sigma_{c,s_1,s_2}) \ge \frac{c+2s_1+s_2}{4}.
 \]
 For our purpose, we define
  \[
  B_{\Sigma}(c,s_1, s_2)
= \frac{1}{4}D(\Sigma_{c,s_1,s_2}) - \frac{s_1+s_2}{4}
\ge \frac{c+2s_1+s_2}{4}  - \frac{s_1+s_2}{4} 
\ge \frac{c}{4}\ge1
 \]
 for all $s_1,s_2$ and for all $c\ge4$.
We want to show $B(\Sigma_{c,s_1,s_2})$ converges. 
To simplify the index notation, denote $\Sigma_{s_1+1}:=\Sigma_{c,s_1+1,s_2}$ the surface obtained by performing an additional move of type 1 on $\Sigma_{c,s_1,s_2}$. 
Zupan already proved in \cite[Lemma 5.1]{zupan2013bridge} that 
\begin{equation}\label{eq-move1-zupan}
D(\Sigma_{s_1+1})\le D(\Sigma_{s_1})+1.
\end{equation}
We now also show that a similar result holds when we perform an additional move of type 2. Denote
$\Sigma_{s_2+1}:=\Sigma_{c,s_1,s_2+1}$ the surface obtained by performing a move of type 2 on $\Sigma_{c,s_1,s_2}$.

\begin{lemma}\label{move 2 invariant non increasing}
    $D(\Sigma_{s_2+1})\leq D(\Sigma_{s_2})+1$.
\end{lemma}
\begin{proof}
Let $P_1$ and $P_2$ be the pants decompositions realizing the distance $D(\Sigma_{s_2})$.  
Consider the puncture $x\in H_1\cap \Sigma_{s_2}$ where an additional type 2 move is performed. It can be viewed as a puncture (i.e., a component) of a pair of pants $P_x$ in the pants decomposition $P_1$.  
There are two possible cases for the other two boundary components of $P_x$: they consist of either a puncture and a pants curve, or two pants curves (see Figure \ref{move2-2cases}).

After the additional type 2 move, the Heegaard surface has one more puncture, which we denote by $\Sigma_{s_2+1}$.  
Due to this additional puncture, the pants decompositions realizing the distance $D(\Sigma_{s_2+1})$ must each have one more curve than $P_1$ and $P_2$, respectively.  
The new puncture, along with the new simple closed curves disjoint from $P_1$, must lie inside $P_x$, as they form a pants decomposition of $\Sigma_{s_2+1}$ with the existing curves (see Figure \ref{move2-2cases}).

Let $P'_1$ be the pants decomposition consisting of the same curves as $P_1$, together with an additional curve lying inside $P_x$.  
Let $P'_x \subset P_2$ be the pair of pants in $P_2$ corresponding to $P_x$.  
Similar as above, let $P'_2$ be the pants decomposition consisting of the same curves as $P_2$, along with an additional curve lying inside $P'_x$. 
This construction ensures that every curve in $P'_1$ bounds a disk or a once-punctured disk in $(H_1, T_1)$, and every curve in $P'_2$ bounds a disk or a once-punctured disk in $(H_2, T_2)$.

\begin{figure}[ht!]
\centering
\labellist 
\pinlabel {$x$} at 340 1050
\pinlabel {$P_x$} at 250 1120
\endlabellist
\includegraphics[width=0.75\linewidth]{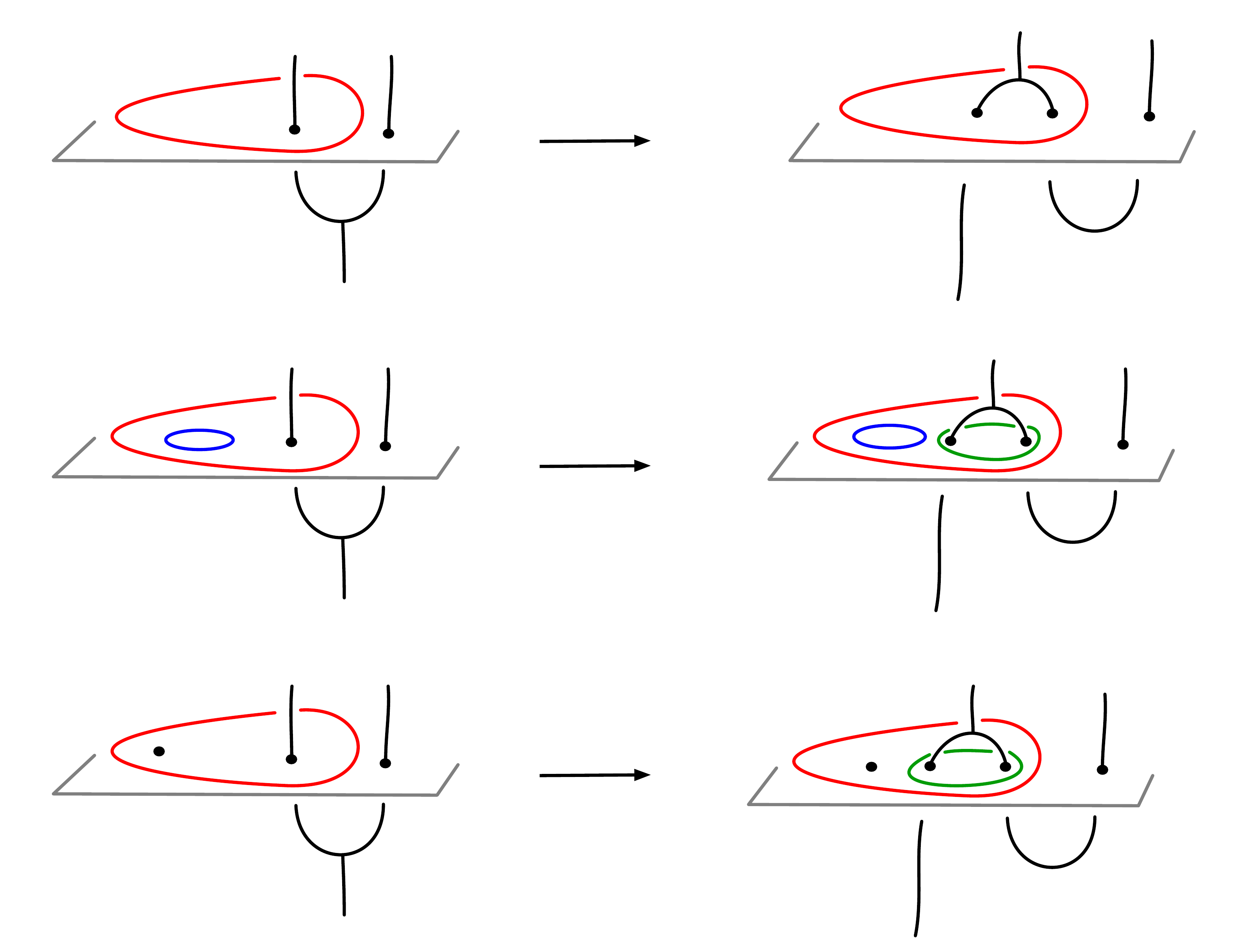}
\caption{Two cases for what happens around the puncture where an additional move of type 2 is performed}
\label{move2-2cases}
\end{figure}

\begin{figure}[ht!]
\centering
\includegraphics[width=0.5\linewidth]{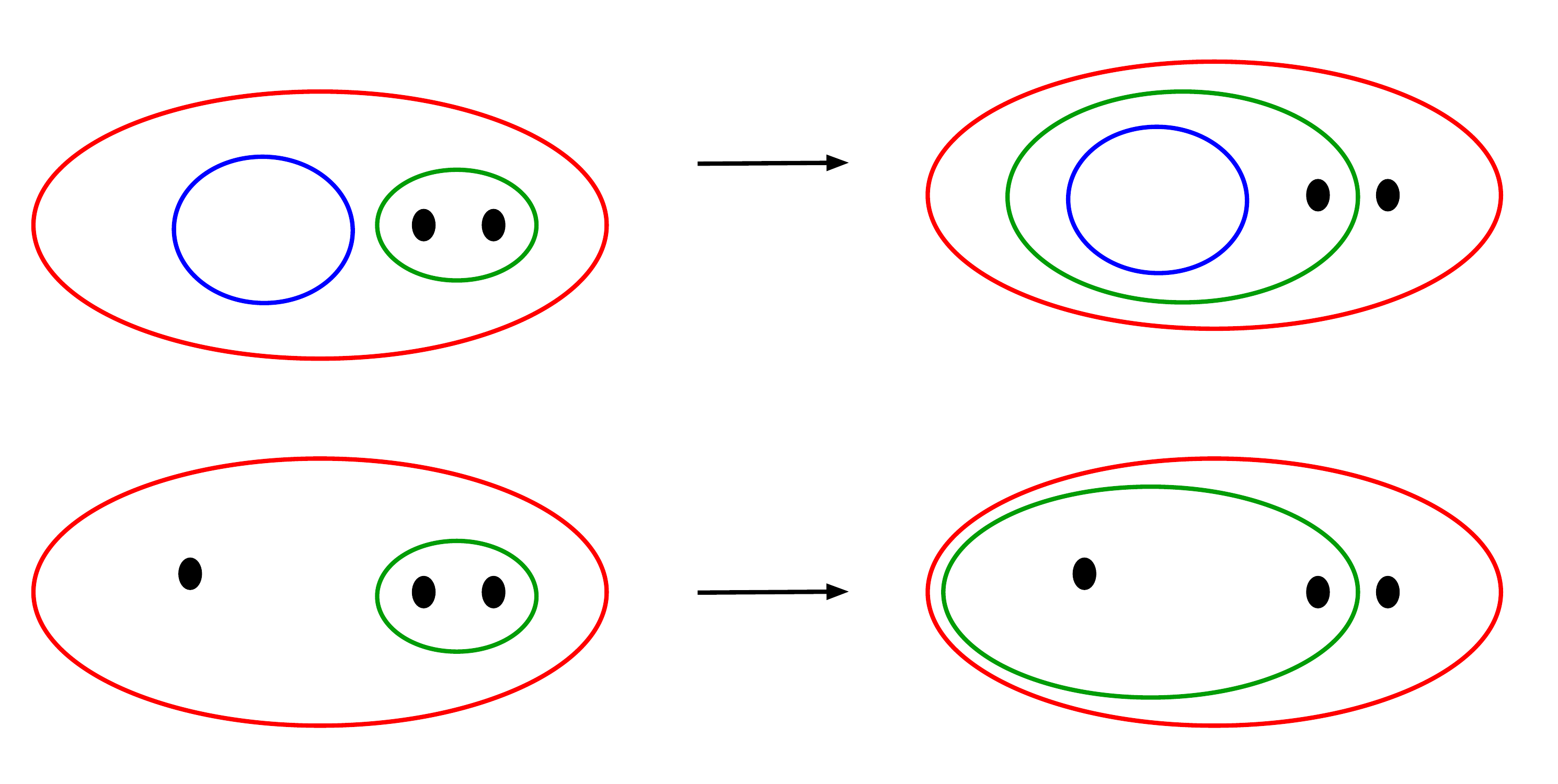}
\caption{An additional step obtained by flipping the new pants curve (the green one)}
\label{fig:additional step}
\end{figure}

Recall that
\[
B_\Sigma(c,s_1,s_2) = \frac{1}{4}D(\Sigma_{c,s_1,s_2}) - \frac{s_1+s_2}{4}.
\]
To simplify notation, let
\[
B(\Sigma_{s_1+1,s_2})
:= B_\Sigma(c,s_1+1,s_2)
\]
and
\[
B(\Sigma_{s_1,s_2+1})
:= B_\Sigma(c,s_1,s_2+1).
\]
To obtain the path from $P'_1$ to $P'_2$, we retain the steps from $P_1$ to $P_2$, and we add an additional step that corresponds to the new curves  (see Figure \ref{fig:additional step}).
\end{proof}

\begin{theorem}
\label{main invariant}
The limit $\lim_{s_1\to\infty, s_2\to\infty} B_\Sigma(c,s_1,s_2)$ exists and is an invariant.      
\end{theorem}

\begin{proof}
    By Lemma \ref{bound below invariant},
    $\{B_\Sigma(c,s_1,s_2)\}_{s_1,s_2}$ is bounded from below by 1.
    Moreover, the sequence is also non-increasing. Indeed, by Lemma \ref{move 2 invariant non increasing} and inequality \eqref{eq-move1-zupan}, we have
    \begin{align*}
    B(\Sigma_{s_1+1,s_2}) &= \frac{1}{4}D(\Sigma_{s_1+1,s_2})-\frac{(s_1+1)+s_2}{4} \\
    & \leq \frac{D(\Sigma_{s_1,s_2})+1}{4}-\frac{(s_1+1)+s_2}{4}\\
    &=B(\Sigma_{s_1,s_2})
\end{align*}  
and
    \begin{align*}
    B(\Sigma_{ s_1,s_2+1}) &= \frac{1}{4}D(\Sigma_{ s_1,s_2+1})-\frac{(s_1+1)+s_2}{4} \\
    & \leq \frac{D(\Sigma_{s_1,s_2})+1}{4}-\frac{s_1 + s_2+1}{4}\\
    &=B(\Sigma_{s_1,s_2})
\end{align*}  
Thus limit $\lim_{s_1\to\infty, s_2\to\infty} B_\Sigma(c,s_1,s_2)$ exists and is an invariant.  
\end{proof}

\section{Future directions and questions}\label{open questions}
 \begin{enumerate}    
     \item Under special conditions on $P, P'$ can we find 3-manifolds with specific properties, like hyperbolic?

     There are some results that connect pants distance to volume of hyperbolic 3 manifolds \cite{Brock2003}. Is it possible to characterize pairs of pants that provide an answer in a converse direction, i.e, starting with some specific $P, P'$ is it possible to construct a 3-manifold?
     \item What about handlebody-knots and links in arbitrary 3-manifold?

     There is not a version of stable equivalence known for knots and links in arbitrary 3-manifolds, it would be interesting to explore the possibility of developing such a theory and trying to define similar invariants. Or is it possible to find invariants of handlebody links in 3-manifolds without a set of stabilization moves?
 \end{enumerate}

\printbibliography

\end{document}